\documentclass[reqno]{amsart}
\usepackage{amsmath,amssymb,amsfonts,amsthm,graphicx,enumerate}
\usepackage{a4wide}
\usepackage{authblk}
\usepackage{enumerate,color}

\usepackage{hyperref}
\usepackage{fancyhdr}

\allowdisplaybreaks

\newtheoremstyle{exampstyle}
  {\topsep} 
  {\topsep} 
  {\itshape} 
  {} 
  {\bfseries} 
  {.} 
  {.5em} 
  {} 

\theoremstyle{exampstyle}

\numberwithin{equation}{section}
\newtheorem{lem}{Lemma}[section]
\newtheorem{theorem}{Theorem}[section]
\newtheorem{defi}{Definition}[section]
\newtheorem{coro}{Corollary}[section]
\newtheorem{prob}{Problem}[section]

\newtheorem{algo}{Algorithm}[section]
\newtheorem{remark}{Remark}[section]
\newtheorem{example}{Example}[section]
\definecolor{orange}{RGB}{255,127,0}
\definecolor{purple}{RGB}{128,0,128}
\definecolor{pink}{RGB}{255,160,150}

\begin{document}

\title[Multi-term Time-Fractional Diffusion Equation]{Strong positivity property and a related inverse source problem for multi-term time-fractional diffusion equations}

\author[X. Yang 
]
{Xiaona Yang}
\address{Xiaona Yang \newline
School of Mathematics and Statistics,
Shandong University of Technology.
266 Xincunxi Road, Zibo, Shandong 255049, China}
\email{xiaonayang1004@163.com}

\author[Z. Li 
]
{Zhiyuan Li}
\address{Zhiyuan Li \newline
School of Mathematics and Statistics,
Shandong University of Technology.
266 Xincunxi Road, Zibo, Shandong 255049, China}
\email{zyli@sdut.edu.cn}

\thanks{Submitted \today.}

\begin{abstract}

In this article, we consider the diffusion equation with multi-term time-fractional derivatives. We first derive that the solution is positive when the initial value is non-negative by a subordination principle for the solution. As an application, we prove the uniqueness of solution to an inverse problem of determination of the temporally varying source term by integral type information in a subdomain. Finally, several numerical experiments are presented to show the accuracy and efficiency of the algorithm.
\end{abstract}

\maketitle

\section{introduction and main results}

In this article, we assume $T$ is a given positive number and $\Omega\subset \mathbb R^d$ is a bounded domain with sufficiently smooth boundary $\partial\Omega$, and we  consider the following initial-boundary value problem:
\begin{equation}
\label{eq-gov}
\left\{
\begin{alignedat}{2}
&\sum_{j=1}^\ell q_j\partial_t^{\alpha_j} (u-u_0) + A(x)u = F(x,t) &\quad& \mbox{in }\Omega\subset \mathbb R^d,\ t\in(0,T),\\
&u(x,0)=u_0,&\quad&\mbox{in }\Omega,\\
&u(x,t)=0,&\quad& \mbox{on }\partial\Omega\times(0,T),
\end{alignedat}
\right.
\end{equation}
where the elliptic operator $A$ is defined for $\psi\in D(A):=\{\psi\in H^2(\Omega);\,\psi=0\mbox{ on }\partial\Omega\}$ as
$$
A\psi(x):=-\sum_{i,j=1}^d\partial_j(a_{ij}(x)\partial_i\psi(x))+c(x)\psi(x),
$$
where we assume $a_{ij}=a_{ji}\in C^1(\overline\Omega)$ ($1\le i,j\le d$), $c(\ge0)\in L^\infty(\Omega)$ and there exists a constant $a_0>0$ such that
$$
\sum_{i,j=1}^da_{ij}(x)\xi_i\xi_j\ge a_0\sum_{i=1}^d\xi_i^2,\quad\forall\,x\in\overline\Omega,\ \forall\,(\xi_1,\ldots,\xi_d)\in\mathbb R^d.
$$
$q_j$, $j=1,2,\cdots,\ell$, are positive constants, $0<\alpha_\ell<\cdots<\alpha_1<1$, and $\partial_t^{\alpha}$ is the $\alpha$-th order Caputo derivative and it should be understood as the inverse of the Riemann-Liouville fractional integral. We will give the details including the definition of the Caputo derivative and the related terminologies in the next section.

The study of such fractional diffusion equations was initially motivated by some physical models. For example, fractional diffusion equations have received great attention in applied disciplines, e.g., in describing some anomalous phenomena including the non-Fickian growth rates, skewness and long-tailed profile which are poorly characterized by the classical diffusion equations (see e.g., Benson, Wheatcraft and Meerschaert \cite{B00}, Levy and Berkowitz \cite{LB03} and the references therein). From the perspective of theoretical research, for example, the Caputo derivative admits a relaxation effect because of its nonlocality in time, which makes long-time decay rate be much slower than the parabolic case (see, e.g., Li, Luchko and Yamamoto \cite{LLY14} and Li, Liu and Yamamoto \cite{LLY15}).
There are also some publications on some important properties showing certain properties similar to classical parabolic equations. For example, a recent result from Li and Yamamoto \cite{LY-FCAA} shows  that the unique continuation principle for one dimensional time-fractional diffusion equation still holds true. We can also see that a maximum principle in the usual setting still holds similarly to the parabolic equation see, e.g., Al-Refai and Luchko \cite{AL} and Luchko \cite{L09} established the maximum principle by a key estimate of the Caputo derivative at an extreme point. Furthermore, Liu, Rundell and Yamamoto \cite{LRY} asserted the positivity property of the solution for the single-term time-fractional diffusion equation but except for a finite set for each $x\in\Omega$.  Later the above positivity property was further generalized to the multi-term case by Liu \cite{Liu17}.  In a recent survey paper Luchko and Yamamoto \cite{LY} the strict positivity of the solution is proved for single-term fractional diffusion equation with strictly elliptic operator. We also refer to Jia, Peng and Yang \cite{JPY} for the strong maximum principle for the diffusion equation involving time-fractional derivative and fractional Laplacian. It reveals that there is no result for the strict  positivity of the solution of the multi-term time-fractional diffusion equation.

In this paper, we will first construct a subordinate principle to solution of parabolic equation for the multi-term time-fractional diffusion equation. Then by means of subordination, we can investigate the strong positivity property of the solution from the subordination identity in view of the results from parabolic case.

\begin{theorem}
\label{thm-sp}
Assume $d\le 3$ and $u\in C([0,T]; L^2(\Omega))\cap C((0,T];H^2(\Omega)\cap H_0^1(\Omega))$ is a solution to the problem \eqref{eq-gov} with homogeneous source term $F=0$ and nonnegative initial value $u_0\ge0,\not\equiv0$ in $\Omega$. Then the inequality $u(x,t)>0$ holds true for any $(x,t)\in\Omega\times(0,T)$.
\end{theorem}

As an application, in equation \eqref{eq-gov} the term can be written as $F(x,t)=g(t)f(x)$. Problems of this type have important applications in several fields of applied science and engineering. For example, the identification of $g(t)$ fits, for example, in the case  of disasters of nuclear power plants, in which the source location can be assumed to be known but the decay of the radiative strength in time is unknown and crucial to be estimated. However, usually this term cannot be directly measured due to the mixing of the effects of several factors, which requires one to use inverse problems to identify these quantities by involving additional information that can be observed or measured practically. For recovering the source term, sometimes we have to measure the interior observation of the solution $u$ since it is difficult to obtain the information on the whole domain at the stage of the diffusion processes. In this paper, our next main concern is:
\begin{prob}
\label{prob-isp}
Assuming $\omega$ is nonempty and open subset of $\Omega$, we consider to determine the unknown source $g$ from the nonlocal integral observation $\int_\omega u(x,t) dx$, $t\in(0,T)$.
\end{prob}

As is known, inverse $t$-source problems for time-fractional diffusion equations are well studied in the literature. Here we do not intend to give a complete list of references, and one can consult Sakamoto and Yamamoto \cite{SY11-JMAA},  Fujishiro and Kian \cite{FK}, Wei, Li and Li \cite{WLL} , Liu, Rundell and Yamamoto \cite{LRY}, Liu and Zhang \cite{LZ}, Ruan and Wang \cite{RW} for example. See also Aleroev, Kirane and Malik \cite{AKM} where the $\omega=\Omega$ but the source term has a more general form: $F(x,t)=\lambda(t)f(x,t)$. For other related works on the inverse $t$-source problems , see also Jin and Rundell \cite{JR}, and Wang and Wu \cite{WW}. It reveals that most of publications on inverse $t$-source problems  for fractional equations are concerned with the symmetric case. Moreover, see for example Huang, Li and Yamamoto \cite{HLY}, Jiang, Li, Liu and Yamamoto \cite{JLLY},  Li and Zhang \cite{LiZ}, Sakamoto and Yamamoto \cite{SY11-MCRF}, Wang, Zhou and Wei \cite{WZW},   Zhang and Xu \cite{ZX} and the references therein for inverse $x$-source problems. 
As far as the author's knowledge, there is no publication dealt with inverse $t$-source problem with multi-term fractional derivatives from integral type observation.

In this paper, for the  equation \eqref{eq-gov}, we show the temporal component can be uniquely determined from the nonlocal integral observation data by the use of the strong positivity of the solution. We have
\begin{theorem}
\label{thm-isp}
Assume $f(x)\in L^2(\Omega)$ is non-negative and not identically vanished in $\Omega$, and suppose $u$ solves the initial-boundary value problem \eqref{eq-gov}. Let $\omega$ be a fixed nonempty  and open subdomain of $\Omega$. Then $g=0$ if $\int_{\omega}u(x,t)dx=0$ for all $t\in(0,T)$.
\end{theorem}

For proving Theorem \ref{thm-sp}, several technical lemmas for completely monotonic functions and analyticity of the solution to the problem with $F=0$ are needed, so we collect them in Section \ref{sec-pre}. Preparing all necessities, we will establish a subordinate principle to solution of parabolic equation, from which will finish the proof of Theorem \ref{thm-sp} in Section \ref{sec-sp}. In the following Section \ref{sec-isp}, an integral identity is to be put forward with which the uniqueness in Theorem \ref{thm-isp} can be proved via the strong positivity of the solution established in the above section.
In Section \ref{sec-num}, an iteration method based on the Tikhonov regularization is designed to obtain the numerical solution for our inverse source problem, and some typical numerical experiments are tested to verified the validity of the iteration scheme.
Finally, concluding remarks are given in Section \ref{sec-rem}.

\section{Preliminaries}
\label{sec-pre}
\subsection{Fractional calculus}
In this part, we first set up notations and terminologies, and review some of standard facts on the fractional calculus.


We define the Riemann-Liouville fractional integral $J^\alpha$ of order $\alpha>0$.
\begin{equation}
\label{defi-RL}
J^\alpha g(t) = \frac1{\Gamma(\alpha)} \int_0^t (t-\tau)^{\alpha-1} g(\tau) d\tau,\quad t>0,
\end{equation}
where $\Gamma(\cdot)$ is the Gamma function. By $\mathcal R(J^\alpha)$ we denote the domain and the range of the operator $J^\alpha:L^2(0,T)\to L^2(0,T)$. It is known the Riemann-Liouville integral operator $J^{\alpha}: L^2(0,T) \to \mathcal R(J^\alpha)$ is bijective, see e.g., Gorenflo, Luchko and Yamamoto \cite{GLY15}. Then  we define the time-fractional derivative $\partial_t^\alpha$ ($0<\alpha<1$) on $\mathcal R(J^\alpha)$ by
$$
\partial_t^\alpha g = J^{-\alpha} g := (J^{\alpha})^{-1}g, \quad g \in \mathcal R(J^\alpha).
$$
Moreover, from the definition of the Riemann-Liouville integral operator, the Caputo derivative $\partial_t^\alpha$ can be rephrased as
$$
\partial_t^\alpha g(t)= \frac{d}{d t}J^{1-\alpha}g(t),\quad g\in \mathcal R(J^\alpha).
$$
Parallelly, we define the backward Riemann-Liouville integral operator $J^\alpha_{T-}$ by
$$
J^\alpha_{T-}g(t)=\frac1{\Gamma(\alpha)}\int_t^T \frac{g(\tau)}{(\tau-t)^{1-\alpha}} d\tau, \quad 0<t<T,
$$
and the backward Caputo derivative by $\partial_{T-}^\alpha g(t):= -\frac{d}{d t}J^{1-\alpha}_{T-}g(t)$, $g\in \mathcal R(J_{T-}^\alpha)$.

We will show several useful lemmata which are related to the above fractional integral and derivative and they will be used in the
forthcoming discussion. The first one is about the convolution formula for the Riemann-Liouville fractional integral.
\begin{lem}
\label{lem-convo-RL}
Let $\alpha>0$ and $g,h\in L^2(0,T)$, then
$$
J^\alpha\left(\int_0^t g(\tau) h(t-\tau) d\tau \right)
=\int_0^t  g(\tau) (J^\alpha h)(t-\tau) d\tau,\quad 0<t<T.
$$
\end{lem}
\begin{proof}
From the definition \eqref{defi-RL} of the Riemman-Liouville fractional integral $J^\alpha$, it follows that
$$
J^\alpha \left(\int_0^t g(\tau) h(t-\tau) d\tau \right)
=\frac1{\Gamma(\alpha)} \int_0^t (t-\tau)^{\alpha-1} \int_0^\tau g(\eta) h(\tau-\eta) d\eta d\tau.
$$
By the Fubini lemma, we see that
$$
J^\alpha \left(\int_0^t g(\tau) h(t-\tau) d\tau \right)
=\frac1{\Gamma(\alpha)} \int_0^t g(\eta)  d\eta\int_\eta^t (t-\tau)^{\alpha - 1} h(\tau-\eta) d\tau.
$$
After the change of the variable, we find
\begin{align*}
J^\alpha \left(\int_0^t g(\tau) h(t-\tau) d\tau \right)
=& \frac1{\Gamma(\alpha)} \int_0^t g(t-\eta) d\eta \int_0^\eta (\eta-\tau)^{\alpha -1} h(\tau) d\tau
\\
=& \int_0^t g(t-\eta) J^\alpha h(\eta) d\eta.
\end{align*}
We complete the proof of the lemma.
\end{proof}

\begin{lem}
\label{lem-semigroup}
Let $g\in L^2(0,T)$ and $\alpha,\beta>0$, then
$$
J_{T-}^\alpha (J_{T-}^\beta g) = J_{T-}^{\alpha+\beta} g.
$$
\end{lem}
\begin{proof}
From the definition of the backward Riemann-Liouville integral, we see that
$$
J_{T-}^\alpha (J_{T-}^\beta g) = \frac1{\Gamma(\alpha)} \int_t^T (\tau-t)^{\alpha-1} \left [ \frac1{\Gamma(\beta)} \int_\tau^T (\eta - \tau)^{\beta-1} g(\eta) d\eta \right] d\tau,
$$
from which we further employ the Fubini lemma to derive that
\begin{align*}
J_{T-}^\alpha (J_{T-}^\beta g) &= \frac1{\Gamma(\alpha)\Gamma(\beta)} \int_t^T  \left[\int_t^\eta  (\tau-t)^{\alpha-1}  (\eta - \tau)^{\beta-1}  d\tau \right] g(\eta)d\eta
\\
&=  \frac1{\Gamma(\alpha)\Gamma(\beta)} \int_t^T \frac{\Gamma(\alpha)\Gamma(\beta)}{\Gamma(\alpha+\beta)} (\eta -t)^{\alpha+\beta-1} g(\eta)d\eta
\\
&= \frac1{\Gamma(\alpha+\beta)} \int_t^T (\eta - t)^{\alpha+\beta -1} g(\eta) d\eta = J_{T-}^{\alpha+\beta} g.
\end{align*}
We finish the proof of the lemma.
\end{proof}

\begin{lem}
\label{lem-igbp}
Let $g,h\in L^2(0,T)$ and $\alpha>0$, then
$$
\int_0^T \left[J^\alpha g(t)\right] h(t) dt = \int_0^T g(t) J_{T-}^\alpha h(t) dt.
$$
\end{lem}
This lemma is derived from Theorem 3.5 in Samko, Kilbas and Marichev \cite{SK93} .
\begin{coro}
Let $\alpha\in(0,1)$, $g\in \mathcal R(J^\alpha)$ and $h\in \mathcal R(J_{T-}^\alpha)$, then
$$
\int_0^T \left[\partial_t^\alpha g(t)\right] h(t) dt = \int_0^T g(t) \partial_{T-}^\alpha h(t) dt.
$$
\end{coro}
\begin{proof}
From the assumptions in this lemma, we can choose $\phi,\psi\in L^2(0,T)$ such that $g=J^\alpha\phi$ and $h=J_{T-}^\alpha \psi$ respectively. Moreover, we note that $\partial_t^\alpha g(t) = \frac{d}{dt} J^{1-\alpha} g$ and then we have
$$
\partial_t^\alpha g(t) = \frac{d}{dt} J^{1-\alpha} \left( J^\alpha \phi \right) =  \frac{d}{dt} \int_0^t  \phi(\tau) d\tau = \phi(t),
$$
from which we further see that
$$
\int_0^T \left[\partial_t^\alpha g(t)\right] h(t) dt = \int_0^T \phi(t) h(t) dt.
$$
By noting the semigroup property of the backward Riemann-Liouville integral in Lemma \ref{lem-semigroup}, we can similarly obtain that
$$
\int_0^T g(t) \partial_{T-}^\alpha h(t) dt = -\int_0^T g(t) \partial_t J_{T-}^{1-\alpha} J_{T-}^\alpha \psi(t) dt = \int_0^T g(t) \psi(t) dt.
$$
On the other hand, noting that $h=J_{T-}^\alpha \psi$, we conclude from Lemma \ref{lem-igbp} that
$$
\int_0^T \phi(t) h(t) dt = \int_0^T \phi(t) J_{T-}^\alpha \psi(t) dt =  \int_0^T \left[ J^\alpha\phi(t) \right]  \psi(t) dt,
$$
which combined with the fact $J^\alpha \phi = g$ implies
$$
\int_0^T \phi(t) h(t) dt =  \int_0^T g(t)  \psi(t) dt.
$$
We complete the proof of the lemma.
\end{proof}

The further properties of the fractional integral and fractional derivative can be found in e.g., Kilbas, Srivastava and Trujillo \cite{K06} and Podlubny \cite{P99}.

\subsection{Completely monotonic function}
In this part, we give the definition of the completely monotonic functions and list several relevant  results.
\begin{defi}
A function f with domain $(0,\infty)$ is said to be completely monotonic $(c.m.)$,
if it possesses derivatives $f^{(n)}(x)$ for all $n=0,1,\cdots$ and if
$$
(-1)^{n}f^{(n)}(x)>0,\quad \forall x>0.
$$
\end{defi}
\begin{lem}
\label{lem-cm1}
If $f(x)$ and $g(x)$ are c.m.,
then $af(x)+bg(x)$ where $a$ and $b$ are nonnegative constants and $f(x)g(x)$ are also c.m.
\end{lem}
\begin{lem}
\label{lem-cm2}
Let $f(x)$ and $g(x)$ be c.m.,
then
\begin{gather*}
f\left(a+b\int_{0}^{x}g(t) dt\right),
\end{gather*}
where $a$ and $b$ are arbitrary constants, also is c.m., in particular,
the following functions are c.m.
$$
f(ax^{\alpha}+b),\quad a\geq0,\ b\geq0\mbox{ and } 0\leq\alpha\leq1,
$$
\end{lem}
The proofs of the above two lemmas can be found on pp. 2-4 in Miller and Samko \cite{MS01}
\begin{lem}
\label{lem-bern}
A necessary and sufficient condition that the function $f(x)$ should be completely monotonic in the interval $0<x<\infty$ is that
$$
f(x)=\int_{0}^{\infty}e^{-xt} d\alpha(t) ,
$$
where $\alpha(t)$ is a non-decreasing function of such a nature that the integral converges for $x>0$.

The above lemma is well known as the Bernstein theorem, and we refer to e.g., Miller and Samko  \cite{MS01}
\end{lem}

\subsection{Forward problem}
Let $L^2(\Omega)$ be a usual $L^2$-space with the inner product $\langle\,\cdot\,,\,\cdot\, \rangle$, $H_0^1(\Omega)$, $H^2(\Omega)$, etc. denote the usual Sobolev spaces. By $H^{\alpha}(0,T)$ we denote the fractional Sobolev space
with the norm
$$
\|\phi\|_{H^{\alpha}(0,T)}
= \left( \| \phi\|_{L^2(0,T)}^2
+ \int^T_0 \int^T_0 \frac{\vert \phi(t) - \phi(\tau)\vert^2}
{\vert t - \tau \vert^{1+2\alpha}} d\tau dt \right)^{\frac12}
$$
(e.g., Adams \cite{Ad}).

In order to show the unique existence and properties of the weak solution of the initial-boundary value problem \eqref{eq-gov}, we must give a suitable interpretation of the Caputo derivative not by the pointwise definition. We follow the way proposed by \cite{GLY15} to understand the Caputo derivative in the framework of Sobolev spaces.  For this, we equipe the function space $\mathcal R(J^\alpha)$ with the norm $\|\phi\|_{\mathcal R(J^\alpha)}:= \|J^{-\alpha} \phi\|_{L^2(0,T)}$. Then according to \cite{GLY15}, the function space $\mathcal R(J^\alpha)$ with the norm $\|\cdot\|_{\mathcal R(J^\alpha)}$ becomes a Hilbert space. We follow the notation used in \cite{GLY15} for this Hilbert space, that is, we denote it as $H_\alpha(0,T)$, and we adopt the notation $\|\cdot\|_{H_\alpha(0,T)}$ for the norm $\|\cdot\|_{\mathcal R(J^\alpha)}$.

It is known $H_\alpha(0,T)\subset H^\alpha(0,T)$. More precisely,  we have
$$
H_\alpha(0,T) =
\begin{cases}
\{ \phi \in H^{\alpha}(0,T); \, \phi(0) = 0\},  &\frac12 < \alpha < 1, \\
\left\{ \phi \in H^{\frac12}(0,T);\, \int^T_0 \frac{\vert \phi(t)\vert^2}{t} dt
< \infty\right\}, &\alpha = \frac12, \\
H^{\alpha}(0,T),  &0 < \alpha < \frac12,
\end{cases}
$$
and
$$
\|\phi\|_{H_\alpha(0,T)} =
\begin{cases}
\| \phi\|_{H^{\alpha}(0,T)}, & 0 < \alpha < 1, \, \alpha \ne \frac12, \\
\left( \| \phi \|^2_{H^{\frac12}(0,T)} + \int^T_0 \frac{\vert \phi(t)\vert^2}{t} dt\right)^{\frac12}, &
\alpha = \frac12,
\end{cases}
$$
see e.g., Gorenflo, Luchko and Yamamoto \cite{GLY15}.

Under the above settings, the problem \eqref{eq-gov} should be understood as
\begin{equation}
\label{eq-gov-weak}
\left\{
\begin{alignedat}{2}
&\sum_{j=1}^\ell q_j \partial_t^{\alpha_j} (u - u_0) + A u = F &\quad& \mbox{in $\Omega\times(0,T)$,}\\
&u(x,\cdot)-u_0(x)\in H_\alpha(0,T), &\quad&\mbox{for almost all $x\in\Omega$,}\\
&u(x,t)=0, &\quad& \mbox{$(x,t)\in\partial\Omega\times(0,T)$.}
\end{alignedat}
\right.
\end{equation}
We assume that $u_0 \in L^2(\Omega)$ and $F\in L^2(0,T;L^2(\Omega))$.  Then from \cite{GLY15}, the above problem \eqref{eq-gov-weak} admits a unique weak solution
\begin{equation*}
\label{rslt-regu}
u \in L^2(0,T;  H^2(\Omega) \cap H^1_0(\Omega))
\end{equation*}
such that $u-u_0 \in H_{\alpha}(0,T;L^2(\Omega))$. We refer also to e.g., Kubica, Ryszewska and Yamamoto \cite{KRY}, Kubica and Yamamoto \cite{KY} and Zacher \cite{Za}.

In the case of $F=0$, we can get some further properties such as $t$-analyticity and asymptptic estimate of the solution to the problem \eqref{eq-gov-weak}.
\begin{lem}
\label{lem-analy}
Let $0<\alpha<1$ and $T>0$ be fixed constants. Assuming that $c(\ge0)\in L^\infty(\Omega)$. Then the initial-boundary value problem \eqref{eq-gov} with $u_0\in L^2(\Omega)$  and $F=0$ admits a unique weak solution $u\in C((0,T];H^2(\Omega)\cap H_0^1(\Omega)) $  such that
$$
\|u(t)\|_{H^2(\Omega)} \le C_Tt^{-{\alpha_1}}\|u_0\|_{L^2(\Omega)},\quad t\in(0,T].
$$
In the case of $T=\infty$, there holds the following large time asymptotic estimate
$$
\|u(t)\|_{H^2(\Omega)} \le Ct^{-{\alpha_\ell}}\|u_0\|_{L^2(\Omega)},\quad \mbox{as $t$ being sufficiently large.}
$$
\end{lem}

The proof of the $t$-analyticity and the asymptotic estimate of the solution can be found in  Li, Huang and Yamamoto \cite{LHY}, Sakamoto and Yamamoto \cite{SY11-JMAA} and Li, Imanuvilov and Yamamoto \cite{LIY16}.

\section{Strong positivity property}
\label{sec-sp}
In this section, we will show the strong positivity property which is one of remarkable properties of fractional diffusion equations, which asserts that if the initial state of a solution to a homogeneous equation is positive for any $x\in\Omega$, then the solution is strongly positive in the whole domain.
For this, in the Section \ref{sec-sub}, we will establish a subordination principle to solution of parabolic equation for the multi-term time-fractional diffusion equation. As an application of this principle, we will finish the proof of the strong positivity of the solution in the Section \ref{sec-sp}.

\subsection{Subordination principle}
\label{sec-sub}
Letting $F=0$, $c\ge0$ and $u_0\in L^2(\Omega)$. Then from the $t$-analyticity of the solution to the problem \eqref{eq-gov}, we see from Lemma \ref{lem-analy} that the solution $u: (0,T)\to H^2(\Omega)\cap H_0^1(\Omega)$ can be analytically extended to $t\in(0,\infty)$, which allows us to employ the Laplace transforms on both sides of the equation in \eqref{eq-gov}.
For this, by taking Laplace transforms with respect to the variable $t$ on both sides of the equation in \eqref{eq-gov}, and noting the formula
$$
\widehat{\partial_t^\alpha\phi}(s) = s^{\alpha}\widehat \phi(s) - s^{\alpha-1} \phi(0),
$$
see e.g., Kubica, Ryszewska and Yamamoto \cite{KRY}, we see that
\begin{equation}
\label{eq-gov-lap}
\left\{
\begin{alignedat}{2}
&\sum_{j=1}^\ell q_js^{\alpha_j}\widehat u(s) + A(x)\widehat u(s)=\sum_{j=1}^\ell q_js^{\alpha_j-1}u_0 &\quad&\mbox{in }\Omega,\\
&\widehat u(s)=0 &\quad&\mbox{on } \partial\Omega,\quad \Re s>0.
\end{alignedat}
\right.
\end{equation}
We further use Fourier-Mellin formulation (e.g., Theorem 4.3 in Schiff \cite{S99}) for the inversion Laplace transform to derive
\begin{equation*}
u(t)=\frac{1}{2\pi i}\int_{s_0-i\infty}^{s_0+i\infty}  \widehat u(s)  e^{st} ds, \quad s_0>0.
\end{equation*}

Moreover, one can shift the path of integration into the integral contour $\gamma(\theta_0)$, that is
\begin{equation*}
\label{eq-fm}
u(t)=\int_{\gamma(\theta_0)}\widehat{u}(s)e^{st}ds,\quad t\in(0,T),
\end{equation*}
where the contour $\gamma(\theta_0)$ is consist of the following three parts
\begin{enumerate}[1)]
\item $\arg s = -\theta_0, \quad |s|\ge1$;
\item $|\arg s| \le \theta_0,\quad |s|=1$;
\item $\arg s = \theta_0,\quad\ \ \, |s|\ge1$.
\end{enumerate}
In fact, we have the following lemma to ensure the above statement.
\begin{lem}
\label{lem-u-integ}
Let $u$ be the solution to the problem \eqref{eq-gov} with $F=0$ and $u_0\in L^2(\Omega)$. Then its Laplace transform $\widehat u$ can be analytically extended to the sector $\{s\in\mathbb C\setminus \{0\}; |\arg s|\le\theta_0\}$. Moreover, $u$ admits the following integral representation
\begin{equation*}
\label{eq-u-integ}
u(t)=\int_{\gamma(\theta_0)} \widehat u(s) e^{st} ds,  0<t\le T.
\end{equation*}
\end{lem}
\begin{proof}
From \eqref{eq-gov-lap} and noting that the Dirichlet eigensystem  $\{\varphi_n\}$ of the operator $A(x)$ forms an othornormal basis of $L^2(\Omega)$, it follows that
\begin{equation*}
\widehat{u}(s)=\sum_{n=1}^{\infty}\langle u_0,\varphi_n\rangle\varphi_{n}\frac{s^{-1}Q(s)}{Q(s)+\lambda_n},\quad \Re s>0.
\end{equation*}
It is not difficult to see that the series on the right-hand side of the above equation can by analytically extended to the sector
$\{s\in\mathbb{C}\setminus\{0\};\vert{\arg s}\vert \le \theta_0\}$, where $\theta_0$ is such that $\frac{\pi}{2} < \theta_0 < \min\{\frac{\pi}{2\alpha_1},\pi\}.$ Indeed, for any $\theta=\vert{\arg s}\vert \le \theta_0,$ we assert that $0 < \theta\alpha_j \le \theta_0\alpha_j < \frac{\pi}{2\alpha_1}\alpha_j < \frac{\pi}{2}$, hence that
$$
\Im (Q(s)+\lambda_n)=\sum_{j=1}^{\ell}\vert{s}\vert^{\alpha_j}\sin{\alpha_j\theta}>0,
$$
from which we further see that
$$
\sum_{n=1}^{\infty} \langle u_0,\varphi_n \rangle \varphi_n\frac{s^{-1}Q(s)}{Q(s)+\lambda_n},\quad \Re s>0
$$
can be analytically extended to the sector
$\{s\in\mathbb{C}\setminus\{0\};\vert{\arg s}\vert \le \theta_0\}$. This means that the Laplace transform $\widehat{u}(s)$ of the solution $u$ to the problem \eqref{eq-gov} with $F=0$ admits the same analytical extention on the sector $\{s\in\mathbb{C}\setminus\{0\};\vert{\arg s}\vert \le \theta_0\}$. We still denote the extention as $\widehat{u}(s)$ if there is no conflict occurs. We now fix $s_0>0$ and we employ the Fourier-Mellin formula to derive
$$
u(t) = \frac1{2\pi i}\int_{s_0-i\infty}^{s_0+i\infty} \widehat{u}(s)e^{st}ds,\quad t\in(0,T).
$$
From the Jordan lemma (e.g., Wei and Zhu \cite{WZ97}) and the Cauchy theorem (e.g., Theorem 1.1 in Chapter 2 in Stein and Shakarchi \cite{St}), we assert that the above path of integration can be shifted to the contour $\gamma(\theta)$, that is
\begin{equation}
\label{eq-fm}
u(t) =  \frac1{2\pi i}\int_{\gamma(\theta_0)}\widehat{u}(s)e^{st}ds,\quad t\in(0,T).
\end{equation}
This completes the proof of the lemma.
\end{proof}

In order to constructing the subordination principle to a solution of parabolic equation for the problem \eqref{eq-gov}, we consider
\begin{equation}
\label{eq-heat}
\left\{
\begin{alignedat}{2}
&\partial_t w+A(x)w=0 &\quad&\mbox{in }\Omega\times(0,\infty),\\
&w(x,0)=u_0 &\quad&\mbox{in }\Omega, \\
&w(x,t)=0 &\quad& \mbox{on }\partial\Omega\times(0,+\infty).
\end{alignedat}
\right.
\end{equation}
Then using the Laplace transform argument, we rephrase the above problem \eqref{eq-heat} as follows
\begin{equation}
\label{eq-heat-lap}
\left\{
\begin{alignedat}{2}
&\eta \widehat w(\eta) + A(x)\widehat w(\eta) =  u_0&\quad&\mbox{in }\Omega,\\
&\widehat w(\eta)=0&\quad& \mbox{on }\partial\Omega,\quad \Re \eta>0.
\end{alignedat}
\right.
\end{equation}
Now we divide both sides of the equation in \eqref{eq-gov-lap} by $\sum_{j=1}^\ell q_js^{\alpha_j-1}$ and we obtain
\begin{equation*}
\left\{
\begin{alignedat}{2}
&\sum_{j=1}^\ell q_js^{\alpha_j}\left[\frac{\widehat u(s)}{\sum_{j=1}^\ell q_js^{\alpha_j-1}}\right] + A(x)\left[\frac{\widehat u(s)}{\sum_{j=1}^\ell q_js^{\alpha_j-1}}\right] = u_0&\quad&\mbox{in }\Omega,\\
& \frac{\widehat u(s)}{\sum_{j=1}^\ell q_js^{\alpha_j-1}}=0&\quad&\mbox{on }\partial\Omega.
\end{alignedat}
\right.
\end{equation*}
Then letting $\eta=\sum_{j=1}^\ell q_js^{\alpha_j}$ and recalling the uniqueness of the solution to the problem \eqref{eq-heat-lap} we must have $\widehat w(\eta)=\frac{\widehat u(s)}{\sum_{j=1}^\ell q_js^{\alpha_j-1}}$, that is

\begin{equation}\label{eq-u-w}
\widehat u(s)=s^{-1}Q(s) \widehat w(Q(s)) \mbox{ with }Q(s):=\sum_{j=1}^\ell q_js^{\alpha_j}, \quad |\arg s|\le \theta_0.
\end{equation}
On the basis of the above relation \eqref{eq-u-w} and Lemma \ref{lem-u-integ}, we can establish the following subordinate principle.
\begin{lem}
\label{lem-sub}
Suppose that $u$ and $w$ are the solutions to the problem \eqref{eq-gov} with $F=0$ and $u_0\in L^2(\Omega)$ and the problem \eqref{eq-heat} respectively. Then $u$ and $w$ admit the following relation
\begin{equation}
\label{eq-sub}
u(t)=\int_0^\infty w(\tau) \left(\frac{1}{2\pi i} \int_{s_0-i\infty}^{s_0+i\infty}  s^{-1}Q(s)  e^{st- Q(s)\tau } ds\right) d\tau,
 \quad s_0>0.
\end{equation}
\end{lem}
\begin{proof}
Noting the relation \eqref{eq-u-w}, we have
$$
\widehat{u}(s)=s^{-1}Q(s)\widehat{w}(Q(s)),\quad s\ne0,  \vert{\arg s}\vert \le \theta_0.
$$
which combined with \eqref{eq-fm} further implies
$$
u(t)= \frac{1}{2\pi i} \int_{\gamma(\theta_0)}s^{-1}Q(s)\widehat{w}(Q(s))e^{st}ds,\quad t\in(0,T).
$$
Finally, from the definition of the Laplace transform, it follows that
\begin{equation*}
\begin{aligned}
u(t)&=\frac{1}{2\pi i} \int_{\gamma(\theta_0)}  s^{-1} Q(s) \left(\int_0^\infty w(\tau) e^{-\tau Q(s)} d\tau\right) e^{st} ds,\\
&=\int_0^\infty w(\tau) \left(\frac{1}{2\pi i} \int_{\gamma(\theta_0)}  s^{-1}Q(s)  e^{st- Q(s)\tau } ds\right) d\tau,
 \quad 0<t\le T,
\end{aligned}
\end{equation*}
where the exchange of the order of the integrals is justified by the Fubini lemma. We finish the proof of the lemma.
\end{proof}

\subsection{Proof of Theorem \ref{thm-sp}}
\label{sec-sp}
In the above subsection, we have proved the subordination principle which combined the solution to the multi-term time-fractional diffusion equation with the solution of the corresponding parabolic type equation involving the same initial and boundary conditions. In order to show the strong positivity property of the solution, it remains to show the following lemma.
\begin{lem}
\label{lem-K}
Let $\alpha_j\in (0,1)$ and $q_j>0$ be constants. Then the kernel function in \eqref{eq-sub}
$$
K_\alpha(t,\tau) =
\frac{1}{2\pi i} \int_{s_0-i\infty}^{s_0+i\infty}  s^{-1}Q(s)  \exp\{st- Q(s)\tau \} ds
$$
is nonnegative for a.e. $t,\tau>0$, where $Q(s)=\sum_{j=1}^\ell q_js^{\alpha_j}$.
\end{lem}
\begin{proof}
It is not difficult to check that the Laplace transform of the function $K_{\alpha}(\cdot,\tau)$ can be calculated as follows
$$
\begin{aligned}
\mathcal L\{K_{\alpha}(\cdot,\tau);s\}&=s^{-1}Q(s)\exp\{-Q(s)\tau\},\\&=\sum_{j=1}^{\ell}q_js^{\alpha_j-1}\prod_{k=1}^{\ell}\exp\{-\tau q_{k}s^{\alpha_k} \}.
\end{aligned}
$$
Therefore, from  the Bernstein theorem in Lemma \ref{lem-bern}, it is sufficient to show that the Laplace transform
$\mathcal L\{K_{\alpha}(\cdot,\tau);s\}$
is completely monotonic,
that is
$$
(-1)^{n}\partial_{s}^{n}\mathcal L\{K_{\alpha}(\cdot,\tau);s\}\geq0,\quad \forall s>0.
$$
Indeed, from  \eqref{eq-sub} in Lemmas \ref{lem-cm1} and \ref{lem-cm2}, it follows that $s^{\alpha_j-1}$,
$\exp\{-\tau q_k s^{\alpha_k} \}$ are completely monotonic functions ,
where $\alpha_j\in(0,1),q_j,\tau>0$.
Consequently,
we derive that
$$
\sum_{j=1}^{\ell}q_{j}s^{\alpha_j-1}\prod_{k=1}^{\ell} \exp\{-\tau q_{k}s^{\alpha_k} \}
$$
is also completely monotonic function.
Finally,
in view of the Bernstein theorem in Lemma \ref{lem-bern}, we must have
$$
K_{\alpha}(t,\tau)\geq0,\quad \forall t,\tau>0.
$$
We finish the proof of the lemma.
\end{proof}

Now we are ready to give the proof of the first main result.
\begin{proof}[Proof of Theorem \ref{thm-sp}]
From Lemma \ref{lem-sub} and noting the definition of $K_\alpha(t,\tau)$, it follows that
$$
u(t)=\int_0^\infty w(\tau) K_\alpha(t,\tau) d\tau,\quad t\in(0,T).
$$
On the other hand, in view of the strong maximum principle for the parabolic equation, we see that the solution $w$ to the problem \eqref{eq-heat} is strongly positive for any $(x,t)$ in $\Omega\times(0,T)$, which combined with Lemma \ref{lem-K} further implies that $u>0$ in $\Omega\times(0,T)$. The proof of the theorem is complete.
\end{proof}

\section{Inverse source problem}
\label{sec-isp}

For proving the second main result, in Section \ref{sec-ii} we will give an integral identity which reflects a corresponding relation of varied of the unknown source functions with the additional observations, and in Section \ref{sec-isp-proof}

\subsection{Integral identity}
\label{sec-ii}
We assume $u$ solves the following initial-boundary value problem \eqref{eq-gov} with $u_0=0$ and $F(x,t)=g(t)f(x)$. Practically, $g(t)$ is always unknown, and we focus on the unique determination of the source term $g(t)$ from the non local observation data
$$
\int_{\omega}u(x,t)dx, \mbox{ where $\omega$ is a nonempty and open subset of $\Omega$}.
$$
We have
\begin{lem}
\label{lem-ii}
Assume $u$ is a solution to the problem \eqref{eq-gov} with $F=g(t)f(x)$ and $u_0=0$. Then the integral type observation $\int_\omega u(x,t) dx$ can be represented by following convolution form.
$$
\sum_{j=1}^\ell q_jJ^{1-\alpha_j} \left(\int_\omega u(x,t) dx \right)
=\int_{0}^{t}g(t-\tau) \left(\int_\omega v(x,\tau)dx \right) d\tau,\quad t\in(0,T),
$$
where $v$ satisfies the following initial-boundary value problem
\begin{equation}
\label{eq-homo}
\left\{
\begin{alignedat}{2}
&\sum_{j=1}^\ell q_j\partial_t^{\alpha_j} v + A(x)v = 0  &\quad& \mbox{ in }\Omega\times(0,T),\\
&v(x,0)=f(x) &\quad& \mbox{ in }\Omega,\\
&v(x,t)=0 &\quad& \mbox{ on } \partial\Omega\times(0,T).
\end{alignedat}
\right.
\end{equation}
\end{lem}
\begin{proof}
From Lemma 4.2 in Liu \cite{Liu17}, we obtain that $u$ allows the representation
\begin{equation}
\label{eq-duhamel}
u(\cdot,t)= \int_0^t \mu(t-\tau) v(\cdot, \tau) d\tau,\quad 0<t\le T,
\end{equation}
where $v$ solves the homogeneous problem \eqref{eq-homo} with $f$ as the initial data, and $\mu$ satisfies
\begin{equation}
\label{eq-g}
\sum_{j=1}^\ell q_j J^{1-\alpha_j} \mu(t) = g(t),\quad 0<t\le T.
\end{equation}
Now we multiply $\sum_{j=1}^\ell q_jJ^{1-\alpha_j}$ on both sides of the above equation \eqref{eq-duhamel} and we use Lemma \ref{lem-convo-RL} to derive that
\begin{align*}
\sum_{j=1}^\ell q_j J^{1-\alpha_j} u(x,t)
&=\int_{0}^t v(\cdot,t-\tau) \sum_{j=1}^\ell q_j J^{1-\alpha_j} \mu(\tau) d\tau
\\
&=\int_{0}^{t}g(t-\tau)v(\cdot,\tau) d\tau,\quad t\in(0,T).
\end{align*}
Here in the last equality, we used the relation \eqref{eq-g}. We finally take integration on the subdomain $\omega$, and from the Fubini lemma, we can get the desired result.
\end{proof}

\subsection{Proof of Theorem \ref{thm-isp}}
\label{sec-isp-proof}
In this part, we will use the integral identity established in Lemma \ref{lem-ii} and the strong positivity of the solution in Theorem \ref{thm-sp} to show the uniqueness of the determination of the source term from the integral type observation.
\begin{proof}[Proof of Theorem \ref{thm-isp}]
From the strictly positive property of the solution to the problem \eqref{eq-gov} with $F=0$ and $u_0\ge0$, it follows that the initial-boundary value problem \eqref{eq-homo} with $f\ge0$ admits a unique solution such that $v(x,t)>0$, $\forall (x,t) \in\Omega\times(0,T)$.

Then noting the obversion$\int_{\omega}u(x,t) dx=0$, $t\in(0,T)$, we conclude from the integral identity in Lemma \ref{lem-ii} that
$$
\int_{0}^{t} g(t-\tau) \left( \int_{\omega}v(x,\tau) dx \right) d\tau=0,\quad \forall t\in(0,T).
$$
Consequently, the Titchmarsh convolution theorem, see e.g., Titchmarsh \cite{T26} implies the existence of $T_1, T_2 \in [0,T]$ satisfying $T_1+T_2 \ge T$ such that $g(t) = 0$ for almost all $t \in (0, T_1)$ and $\int_\omega v(x,t)dx = 0$ for almost all $t \in [0, T_2]$. However, Theorem \ref{thm-sp} asserts that $\int_\omega v(x,t)dx >0$ is valid for any $(x,t)\in\Omega\times(0,T)$. As a result, the only possibility is $T_1=0$ and thus $T_2=T$ , that is, $g = 0$ in $(0, T )$, which finishes the proof of the theorem.
\end{proof}

\section{Numerical Simulation}
\label{sec-num}
In this section, we are devoted to developing an effective numerical method for the numerical reconstruction of the unknown source in $(0,T)$ from the addition data $\int_\omega u(x,\cdot) dx$ in $(0,T)$.

\subsection{Iterative thresholding algorithm}
\label{sec-ite}

We discuss the problem \eqref{eq-gov} with $F=g(t)f(x)$ and we write its solution of problem \eqref{eq-gov} as $u[g]$ in order to emphasize the dependency on the unknown function $g$. Here and henceforth, we set $g_{\rm true}\in L^{2}(0,T)$ as the true solution to our inverse source problem \ref{prob-isp}. By using noise contaminated observation data $E^{\delta}(t):=\int_\omega u(x,t) dx$ in $(0,T)$, where $E^{\delta}$ satisfies $\|E^{\delta}-\int_\omega u[g_{\rm true}]\|_{L^2(0,T)}\le \delta$ with the noise level $\delta$, we carry out numerical reconstruction.

In the framework of the Tikhonov regularization technique, we propose the following output least squares functional related to our inverse source problem{\color{red}:}
$$
\Phi(g)=\frac{1}{2} \left\|\int_{\omega}u[g]dx-E^{\delta}(t) \right\|_ {L^2(0,T)}^{2}+\frac12\lambda \|g\|_{L^2(0,T)}^{2},
$$
where $\lambda>0$ is the regularization parameter.

Now we intend to calculate the Fr\'{e}chet derivative $\Phi'(g)$ of the objective functional $\Phi(g)$ for finding a minimizer. Then for any direction $\xi\in L^2(0,T)$, we see that the $\Phi'(g)$ can be calculated by
 \begin{equation*}
 \begin{aligned}
 \Phi'(g)\xi&=\lim_{\varepsilon \to0}\frac{\Phi(g+\varepsilon\xi)-\Phi(t)}{\varepsilon},\\
 &=\int_{0}^{T}\left[\int_{\omega}u[g](x,t),dx-E^{\delta}(t)\right] \int_{\omega}u'[g]\xi dxdt+\lambda\int_{0}^{T}g\xi dt,
 \end{aligned}
 \end{equation*}
 where $u'[g]\xi$ denotes the Fr\'{e}chet derivative of $u[g]$ in the direction $\xi$. Moreover the linearity of \eqref{eq-gov} immediately yields
$$
u'[g]\xi=\lim_{\epsilon\rightarrow0}\frac{u[g+\epsilon \xi]-u[g]}{\epsilon}=u[\xi].
$$
Indeed, from the notation of $u[g]$, we see that $u[g]$ and $u[g+\varepsilon\xi]$  satisfy
$$
 \sum_{j=1}^\ell q_j\partial_t^{\alpha_j} u[g]+A(x)u[g]=g(t)f(x),
$$
and
$$
\sum_{j=1}^\ell q_j\partial_t^{\alpha_j} u[g+\varepsilon\xi]+A(x)u[g+\varepsilon\xi]=(g(t)+\varepsilon\xi)f(x).
$$
Then
$$
\sum_{j=1}^\ell q_j\partial_t^{\alpha_j} \frac{u[g+\varepsilon\xi]-u[g]}{\varepsilon}+A(x)\frac{u[g+\varepsilon\xi]-u[g]}{\varepsilon}=\xi f.
$$
Letting $\varepsilon\to0$, it follows that
$$
\sum_{j=1}^\ell q_j\partial_t^{\alpha_j} u'[g]\xi + A(x)u'[g]\xi = \xi f,
$$
then we get the desired result: $u'[g]\xi=u[\xi]$.
 Then
 $$
 \Phi'(g)\xi=\int_{0}^{T}\left[\int_{\omega}u[g] dx- E^{\delta}(t)\right] \int_{\omega}u[\xi] dx dt+\lambda \int_{0}^{T}g\xi dt.
$$
 \begin{remark}
It is not applicable to find the minimizer of the functional $\Phi$ directly in terms of the above formula of the Fr\'echet derivative of $\Phi(g)$. Indeed, in the computation for $\Phi'(g)$, one should solve system \eqref{eq-gov} for $u[g]$ with $\xi$ varying in $L^2(0,T)$, which is undoubtedly quite hard and computationally expensive.
\end{remark}

 For this, we introduce the dual system of \eqref{eq-gov} to reduce the computational costs for the Fr\'{e}chet derivatives, that is, the following system for a backward differential equation
 \begin{equation}
\label{eq-dual}
\left\{
\begin{alignedat}{2}
&-\sum_{j=1}^\ell q_j\partial_T^{\alpha_j} w+A(x)w=\chi_\omega \times\left[\int_{\omega}u[g] dx- E^{\delta}(t)\right] &\quad& \mbox{ in }\Omega\times(0,T),\\
&J^{1-\alpha_1}w(T)=0, &\quad& \mbox{ in }\Omega,\\
&w(x,t)=0, &\quad& \mbox{ on }\partial\Omega\times(0,T).\\
\end{alignedat}
\right.
\end{equation}
We multiply $u[\xi]$ on both sides of the above equation \eqref{eq-dual} and take integration on $\Omega\times(0,T)$,
and then we obtain
$$
\begin{aligned}
\int_{0}^{T} \int_{\Omega}\left(-\sum_{j=1}^\ell q_j\partial_T^{\alpha_j} w+A(x)w\right)u[\xi](x,t) dx dt=\int_{0}^{T} \left[ \int_{\omega} u[g] dx-E^{\delta}(t) \right] \left[ \int_{\omega}u[\xi] dx \right] dt.
\end{aligned}
$$
By integration by parts and Lemma \ref{lem-igbp}, it is not difficult to see that
\begin{align*}
\int_{0}^{T}\int_{\Omega}-\sum_{j=1}^\ell q_j\partial_T^{\alpha_j} w u[\xi] dx dt&=\int_{0}^{T}\int_{\Omega}\sum_{j=1}^\ell q_j\partial_t^{\alpha_j} u[\xi] w dx dt
\\
&=\int_{0}^{T}\int_{\Omega} w A(x) u[\xi] dx dt,
\end{align*}
which implies that
\begin{equation*}
\begin{aligned}
\int_{0}^{T}\left[\int_{\omega}u[g] dx-E^{\delta}(t)\right]\int_{\omega}u[\xi] dx dt&=\int_{0}^{T}\left[\int_{\Omega}\sum_{j=1}^\ell q_j\partial_t^{\alpha_j}u[\xi]+L(x)u[\xi] \right] w dx dt\\
&=\int_{0}^{T}\xi(t) \left[ \int_{\Omega}f(x) w dx \right] dt.
\end{aligned}
\end{equation*}
Consequently, we show that
$$
\Phi'(\rho)\xi=\int_{0}^{T} \xi(t) \left[ \int_{\Omega}f w dx \right] dt + \lambda \int_{0}^{T}g\xi dt.
$$
Since $\xi$ can be arbitrarily chosen in $L^2(0,T)$, then we see that
\begin{equation}
\label{eq-phi'}
\Phi'(g)=\int_{\Omega}f(x) w(x,t) dx+\lambda g(t),\quad t\in (0,T),
\end{equation}
where $w$ satisfies \eqref{eq-dual}. Then we can propose conjugate gradient method to reconstruct $g$ by iteration.

We propose an iteration scheme by using the conjugate gradient (CG) method for generating the minimize of $\Phi'(g)$ numerically.
We approximate $g(t)$ by the following iterative process:
$$
g_{k+1}=g_{k}+r_{k}d_{k},\quad k=0,1,\cdots
$$
for suitably chosen step size $r_{k}>0$ and initial guess $g_{0}$,
where $d_k$ is the iterative direction by
$$
d_k=
\left\{
\begin{alignedat}{2}
&-\Phi'(g_{0}), &\quad& \text{if}\quad k=0,\\
&-\Phi'(g_{k})+s_{k}d_{k-1}, &\quad& \text{if}\quad k>0\\
\end{alignedat}
\right.
$$
with
$$
s_k=\frac{\left\| \Phi'(g_{k}) \right\|_{L^2(0,T)}^{2}}{\left\| \Phi'(g_{k-1}) \right\|_{L^2(0,T)}^{2}},\quad r_k =\arg \min_{r\geq0}\Phi'(g_{k}+rd_{k}).\\
$$
Since the operator G:
$g\longmapsto\int_{\omega}u[g] dx$ is linear,
we have
$$
G(g_{k}+r_{k}d_{k})=G(g_{k})+r_{k}G(d_{k}).
$$
Then there holds
$$
\begin{aligned}
\Phi(g_{k}+r_{k}d_{k})=&\frac{1}{2}\Big[\left\| G(g_{k})-E^{\delta} \right\|_{L^2(0,T)}^{2} +r_{k}^{2}\left\| G(d_{k}) \right\|_{L^2(0,T)}^{2}+
\\& +r_{k}\left \langle G(g_{k})-E^{\delta}, G(d_{k})\right \rangle_{L^2(0,T)}
\\& +\frac{\lambda}{2}\left(\left\| g_{k} \right\|_{L^2(0,T)}^{2}+r_{k}^{2}\left\| d_{k} \right\|_{L^2(0,T)}^{2}+2r_{k}\langle g_{k},d_{k}\rangle_{L^2(0,T)}^{2}\right)\Big].\\
\end{aligned}
$$
In order to determine the step size $r_{k}$, we let $\frac{d}{dr}\Phi(g_{k}+rd_{k})=0$ and we see that
$$
r_{k}=-\frac{\left\langle G(g_{k})-E^{\delta},G(d_{k})\right\rangle_{L^2(0,T)}+\lambda\langle_{k},d_{k}\rangle_{L^2(0,T)}}{\left\| G(d_{k}) \right\|_{L^2(0,T)}^{2}+\lambda\left\| d_{k} \right\|_{L^2(0,T)}^{2}}.
$$

Based on the above discussion, we summarize the CG method for reconstructing the $g(t)$ as follows:
\begin{algo}
\label{algori-itera}
Choose a tolerance $\varepsilon>0$, a regularization parameter $\lambda>0$.
\begin{enumerate}
\item[Step 1:] Set $k=0$, the initial guess $g_0$;

\medskip

\item[Step 2:]  Compete $d_{0}(t)=-\Phi'(g)$;

\medskip

\item[Step 3:] Compute the step size $r_{0}>0$ and update $\beta_{1}=\beta_{0}+r_0d_{0}$;

\medskip

\item[Step 4:] For $k=1,w,\cdots$, compute $s_{k}$, $d_{k}$ and $r_{k}$;

\medskip

\item[Step 5:] Update $g_{k+1}=g_{k}+r_{k}d_{k}$.
If a stoping criterion is satisfied,
output $g_{k+1}$ and stop.
Otherwise,
set $k+1\Longrightarrow k$ and go to step 4.

\end{enumerate}
\end{algo}

\begin{remark}
As can be seen from \eqref{eq-phi'}, for computing $\Phi'[g_k]$ at each iteration step, one only needs to solve the forward problem \eqref{eq-gov} once and the backward problem \eqref{eq-dual} once. Therefore, the numerical implementation of Algorithm \ref{algori-itera} is easy and computationally cheap.
\end{remark}

\subsection{Numerical experiments}
\label{sec-numer}
In this part, we set $T=1$, $\Omega=(0,1)$ and $f(x)=\sin\pi x$, $x\in\Omega$, and apply the CG algorithm established in the previous subsection to numerically recovery the unknown source term. We carry out several test numerical experiments to check the performance of the reconstruction method.

We divide the space-time region $[0,1]\times[0,1]$ into $50\times50$ equidistant meshes. First we set the tolerance parameter $\varepsilon<0.001$, $\lambda=0.00001$, initial guess $g_{0}(t)=0$. We consider the noisy data generated in the form
$$
E^{\delta}(t)=(1+\delta \mathrm{rand}(-1,1)) \int_\omega u[g_{\rm true}](x,t)dx, \quad t\in (0,1),
$$
where rand$(-1,1)$ denotes the uniformly distributed random number in $[-1,1]$ and the noisy level are $\delta=0.5\%$, and $\delta=1.0\%$. We will test the performance of the algorithm with the following examples.

\begin{example}
\label{ex1}
We set $\omega=(0.4,0.6)$ and we consider numerical experiment 
with noise contaminated observation data $\delta=0.1\%$.\\
$(A)$ $\alpha=0.2$, $g_{\rm true}(t)=10t(1-t)$;\\
$(B)$ $\alpha=0.8$, $g_{\rm true}(t)=10t(1-t)$.
\end{example}

In Fig. 1, the source term is obtained after $k=4$ iteration steps, and the relative error of the reconstructed solution of the inverse source problem is less than $0.30\%$. In Fig. 2, the iteration steps $k=5$ and the relative error is less than $0.59\%$. Moreover, we can also see that the fractional order may produce a relatively large impact on the reconstruction results. It is not difficult to find that the recovered curve for the case $\alpha=0.2$ is better than that of the case $\alpha=0.8$, particularly near the integer oder $\alpha=1$. It reveals that the inverse source problem for fractional diffusion equation with higher order has much more ill-posedness than the lower order case.

\begin{example}
\label{ex2}
We set $\omega=(0.4,0.6)$ and we consider numerical experiment 
with noise contaminated observation data $\delta=1\%$.\\
$(A)$ $\alpha=0.2$, $g_{\rm true}(t)=10t(1-t)$;\\
$(B)$ $\alpha=0.8$, $g_{\rm true}(t)=10t(1-t)$.
\end{example}
The results are shown in Figs. 3 and 4, where the noisy level $\delta=1.0\%$. The relative error of the reconstructed source of the inverse problem is less than $4.3\%$ and $6.4\%$ respectively.

Similarly, the reconstruction of $g_{\mathrm true}(t)=1 - |2t-1|$ from the noisy data $E^\delta$ with noisy level $\delta=1\%$ is also performed.
\begin{example}
\label{ex3}
We consider numerical experiment with non-differentiable source.\\
$(A)$ $\alpha=0.2$, $g_{\rm true}(x)=1 - |2t-1|$;\\
$(B)$ $\alpha=0.8$, $g_{\rm true}(x)=1 - |2t-1|$.
\end{example}
The results are shown in Figs. 5 and 6, where the noisy level $\delta=1.0\%$. The source is recovered with the relative error less than $7.8\%$ and $8.5\%$ respectively.

\begin{remark}
Figs 1-4 and the relative errors indicate the efficiency and accuracy of the proposed  Algorithm \ref{algori-itera} for reconstructing the unknown source term. However, it should be mentioned here that the performance of the algorithm is not well near the left boundary $x=0$ for large $\alpha$. This can be solved by taking the observation domain near the left boundary point $x=0$, e.g., $\omega=(0,a)\cap (b,1)$. Moreover, owing to the poor regularity of the target function, the reconstruction results are not as well as those of the smooth case in the previous examples.
\end{remark}

\begin{figure}[htbp]
\centering
\begin{minipage}[t]{0.48\textwidth}
\centering
\includegraphics[width=7cm]{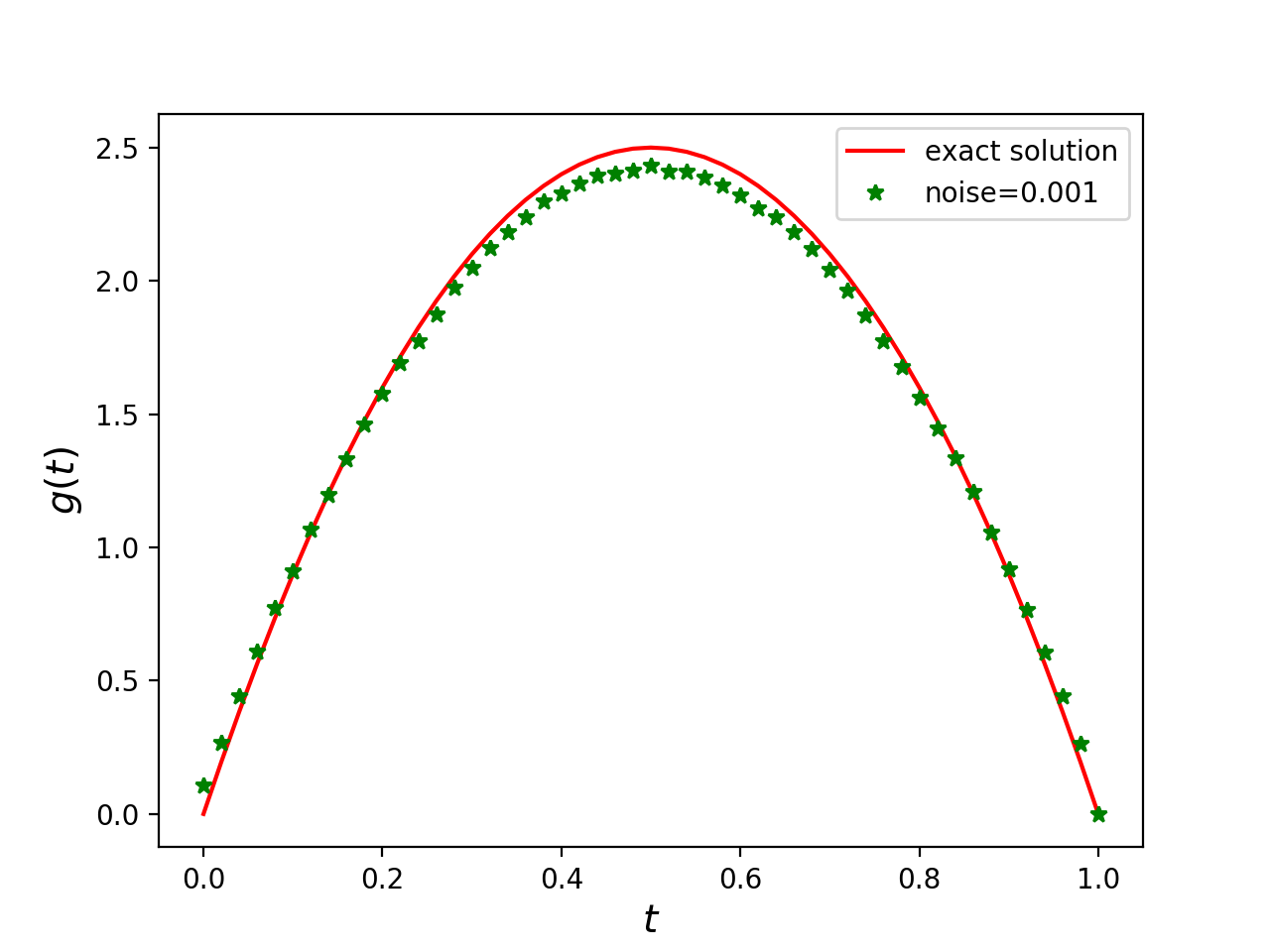}
\caption{Example \ref{ex1} (A): numerical results for different iteration times, when $\alpha=0.2$.}
\end{minipage}
\begin{minipage}[t]{0.48\textwidth}
\centering
\includegraphics[width=7cm]{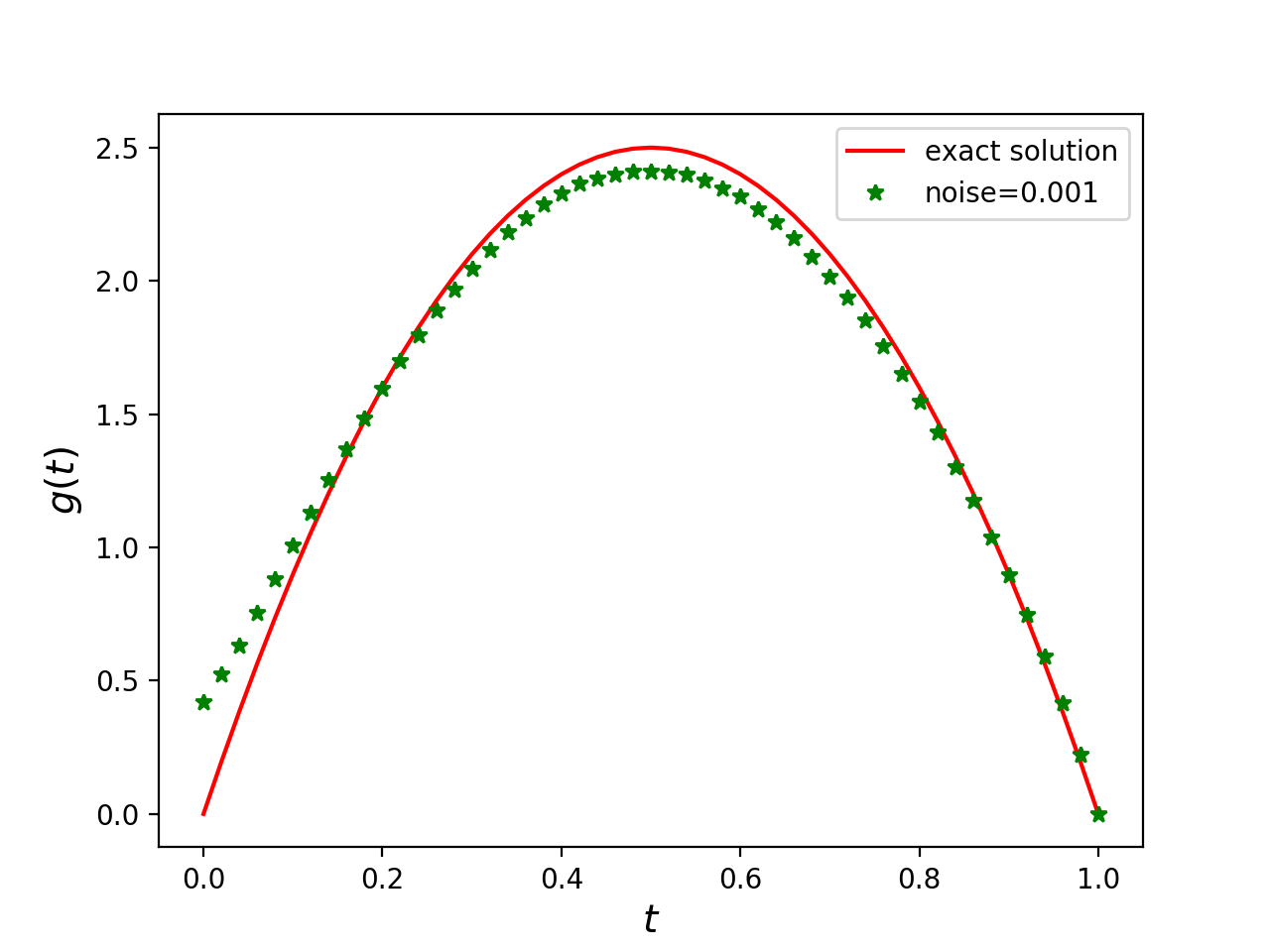}
\caption{Example \ref{ex1} (B): numerical results for different iteration times, when $\alpha=0.8$.}
\end{minipage}
\end{figure}

\begin{figure}[htbp]
\centering
\begin{minipage}[t]{0.48\textwidth}
\centering
\includegraphics[width=7cm]{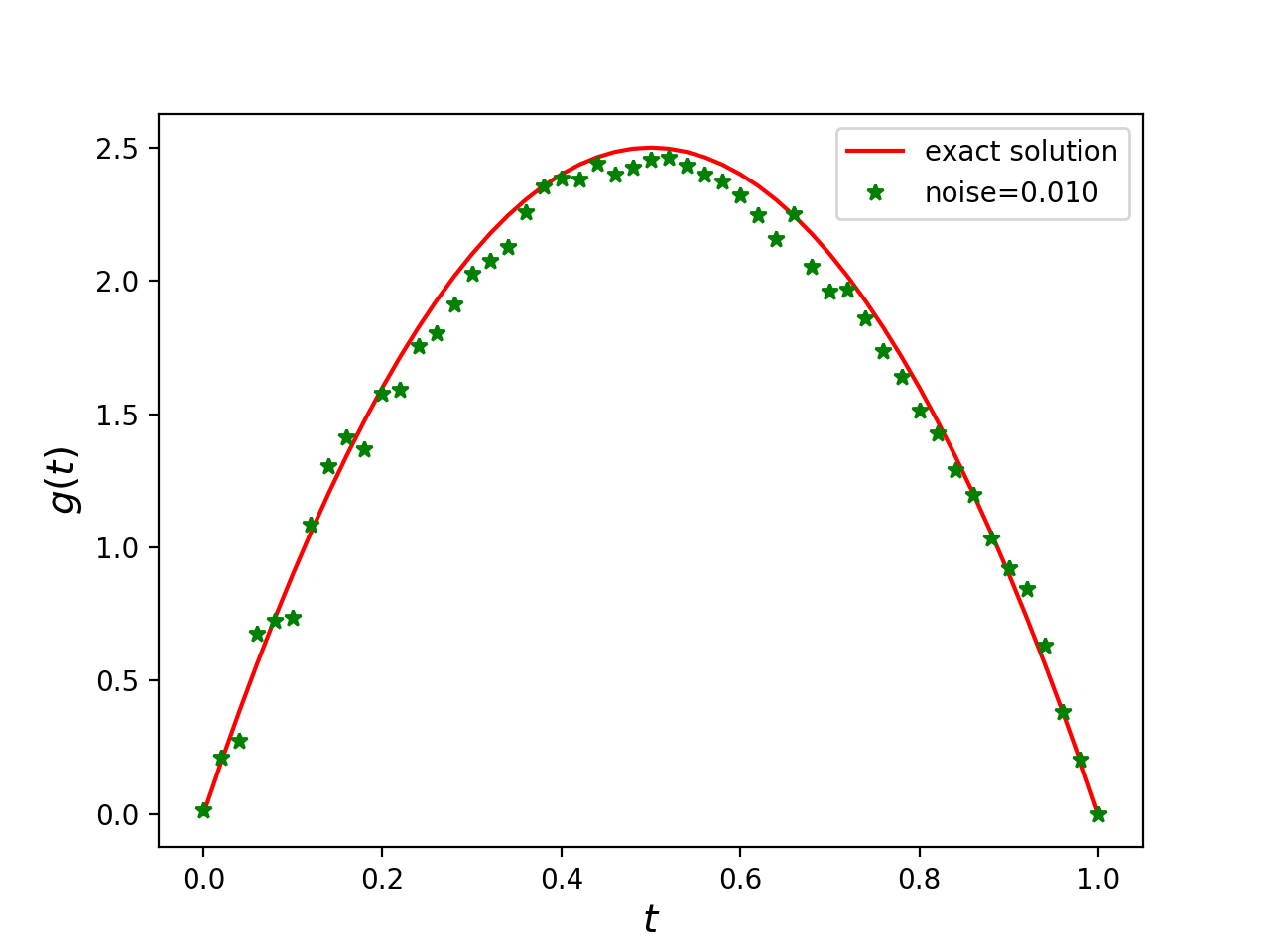}
\caption{Example \ref{ex2} (A): numerical results for different noise level, when $\alpha=0.2$.}
\end{minipage}
\begin{minipage}[t]{0.48\textwidth}
\centering
\includegraphics[width=7cm]{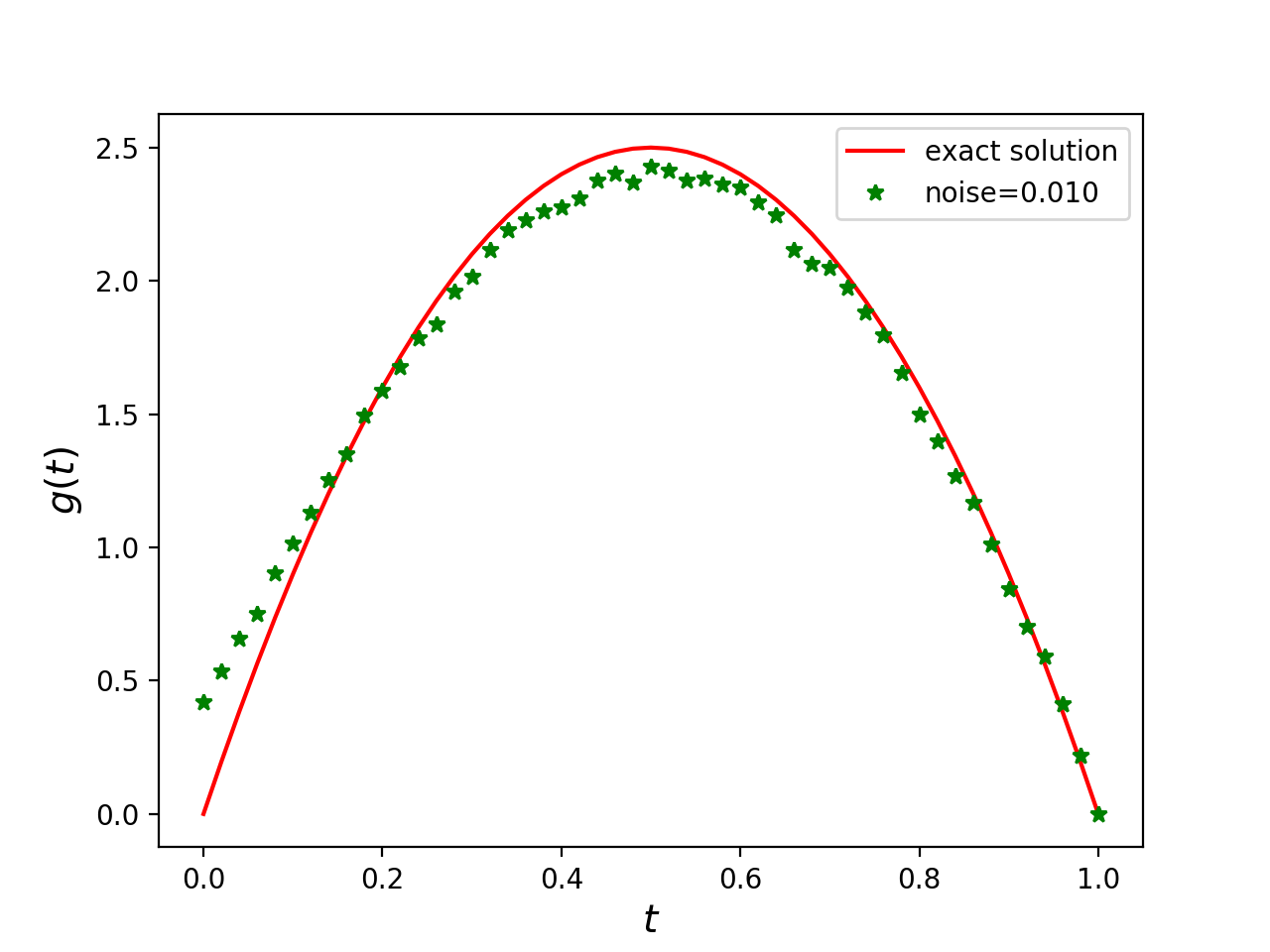}
\caption{Example \ref{ex2} (B): numerical results for different noise level, when $\alpha=0.8$.}
\end{minipage}
\end{figure}

\begin{figure}[htbp]
\centering
\begin{minipage}[t]{0.48\textwidth}
\centering
\includegraphics[width=7cm]{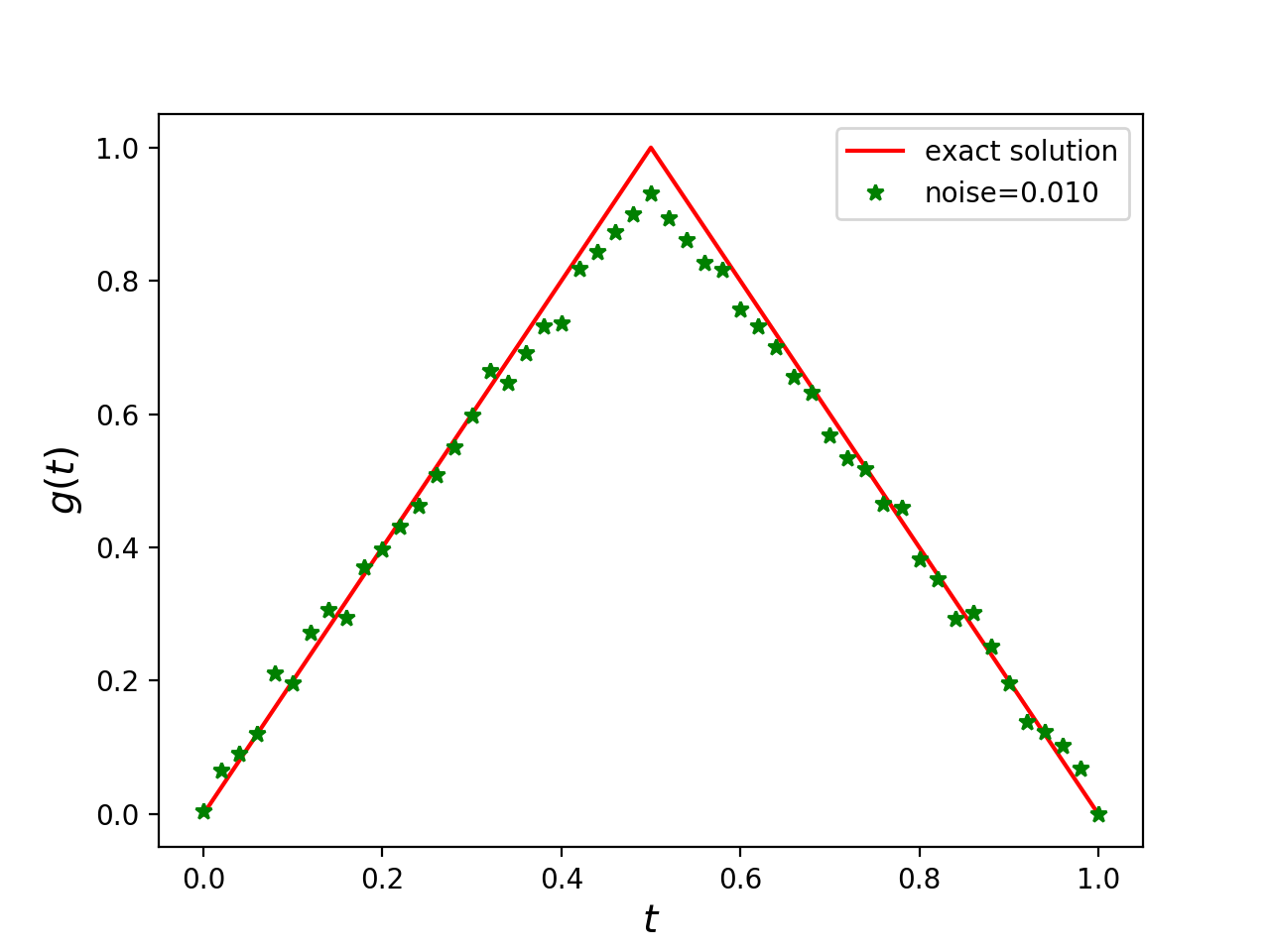}
\caption{Example \ref{ex3} (A): numerical results for different noise level, when $\alpha=0.2$.}
\end{minipage}
\begin{minipage}[t]{0.48\textwidth}
\centering
\includegraphics[width=7cm]{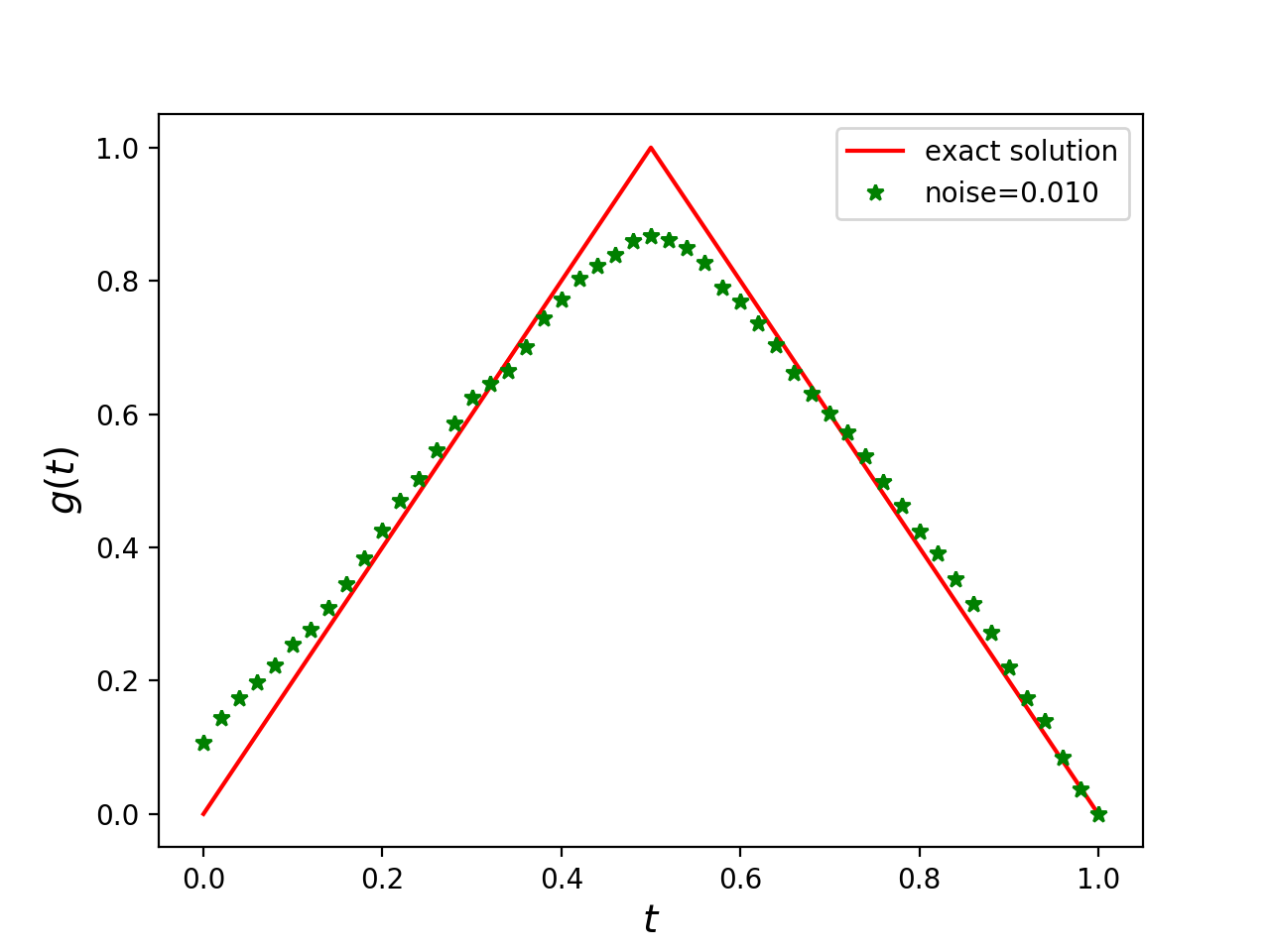}
\caption{Example \ref{ex3} (B): numerical results for different noise level, when $\alpha=0.8$.}
\end{minipage}
\end{figure}

\section{Concluding remarks}
\label{sec-rem}
In this paper, we considered the inverse problem in reconstructing the source term for the multi-term time-fractional diffusion equations from the integral type observation. By the Laplace transform argument, the subordinate identity which gives an integral representation of the solution operator in terms of the corresponding solution to the parabolic equation and a probability density function was established. It will be also interesting to consider the subordination principle that connects the fundamental solution to the problem \eqref{eq-gov} with the solution of the conventional wave equation, see e.g., Bazhlekova and Bazhlekov \cite{BB18}, from which we can further investigate the inverse coefficient problem by following the argument used in Miller and Yamamoto \cite{MY13}. This is another issue which will be considered in next papers. As a direct application of the principle, we next showed that the strong positivity property of the solution to the problem \eqref{eq-gov} with $F=0$ and $u_0\ge0$. Finally, we proved the source term can be uniquely determined from the integral type observation. There will be a challenge if the the diffusion equation has flux. Moreover, in the proofs of our results, we need the assumption that all the coefficients are only $x$-dependent. It will be more interesting and challenging to consider what happens with the properties of the solutions in the case where the coefficients are both $t$- and $x$- dependent. It will be also interesting to consider whether the uniqueness holds true or not for the inverse source problem if $f\ge0$ in $\Omega$ is not valid.

In the numerical aspect, we reformulated the inverse source problem as an optimization problem with Tikhonov regularization. After the derivation of the corresponding variational equation, we characterized the minimizer by employing the associated backward fractional diffusion equation, which results in the iterative method. Then several numerical experiments for the reconstructions were implemented to show the efficiency and accuracy of the proposed Algorithm \ref{algori-itera}. Here we should mention that for finding the minimizer of the inverse source problem is not suitable for $f$ which is not vanished on the boundary. Indeed, for deriving the Algorithm \ref{algori-itera}, the homogeneous boundary condition of the solution $w$ to the backward problem was assumed. It will be interesting to derive the iteration scheme without assuming this homogeneous boundary condition on $f$. The algorithm for the general case remains open.


\section*{Acknowledgments}

The second author thanks National Natural Science Foundation of China 11801326.

\end{document}